\documentclass[11pt,a4paper]{amsart}

\usepackage{amsmath,amssymb,amsthm,amscd,verbatim,enumerate}
\usepackage{graphicx,subfigure,color}

\definecolor{darkred}{rgb}{0.6,0,0}
\definecolor{darkblue}{rgb}{0,0,0.6}
\usepackage[colorlinks=true,citecolor=darkred,linkcolor=darkblue,urlcolor=blue]{hyperref}
\usepackage[lmargin=31mm,rmargin=31mm,bmargin=31mm,tmargin=31mm]{geometry}

\newcommand{\bigh}{\mathcal{H}}

\theoremstyle{plain}
\newtheorem{theorem}{Theorem}[section]
\newtheorem{lemma}[theorem]{Lemma}
\newtheorem{corollary}[theorem]{Corollary}

\theoremstyle{definition}

\title[General Positions Subsets and Independent Hyperplanes]{\boldmath General Position Subsets and\\ Independent Hyperplanes in $d$-Space}

\author{Jean Cardinal, Csaba D. T\'oth, and David R. Wood}
\thanks{Research of Wood is supported by the Australian Research Council.}
\begin{document}
\maketitle
\sloppy

\begin{abstract}
Erd\H{o}s asked what is the maximum number $\alpha(n)$ such that every set of $n$ points in the plane with no four on a line contains $\alpha(n)$ points in general position.
We consider variants of this question for $d$-dimensional point sets and generalize previously known bounds. In particular, we prove the following two results for fixed $d$:
\begin{itemize}
\item Every set $\bigh$ of $n$ hyperplanes in $\mathbb{R}^d$ contains a subset $S\subseteq \bigh$ of size at least $c \left(n \log n\right)^{1/d}$, for some constant $c=c(d)>0$, such that no cell of the arrangement of $\bigh$ is bounded by hyperplanes of $S$ only.
\item Every set of $cq^d\log q$ points in $\mathbb{R}^d$, for some constant $c=c(d)>0$, contains a subset of $q$ cohyperplanar points or $q$ points in general position.
\end{itemize}
Two-dimensional versions of the above results were respectively proved by Ackerman et al. [\emph{Electronic J. Combinatorics}, 2014] and by Payne and Wood [\emph{SIAM J. Discrete Math.}, 2013].
\end{abstract}

\section{Introduction}

\paragraph{\bf Points in general position.}
A finite set of points in $\mathbb{R}^d$ is said to be in \emph{general position} if no hyperplane contains more than $d$ points. Given a finite set of points $P\subset \mathbb{R}^d$ in which at most $d+1$ points lie on a hyperplane, let $\alpha (P)$ be the size of a largest subset
of $P$ in general position. Let $\alpha (n, d) = \min \{ \alpha (P): |P|=n\}$.

For $d=2$, Erd\H{o}s~\cite{Erd86} observed that $\alpha(n,2)\gtrsim\sqrt{n}$ and proposed the determination of $\alpha(n,2)$ as an open problem\footnote{We use the shorthand notation $\lesssim$ to indicate inequality up to a constant factor for large $n$. Hence $f(n)\lesssim g(n)$ is equivalent to $f(n)\in O(g(n))$, and $f(n)\gtrsim g(n)$ is equivalent to $f(n)\in \Omega (g(n))$.}. F\"uredi~\cite{F91} proved $\sqrt{n\log n} \lesssim \alpha (n,2)\leq o(n)$, where the lower bound uses independent sets in Steiner triple systems, and the upper bound relies on the density version of the Hales-Jewett Theorem~\cite{FK89,FK91}. F\"uredi's argument combined with the quantitative bound for the density Hales-Jewett problem proved in the first polymath project~\cite{polymath} yields $\alpha (n, 2) \lesssim n / \sqrt{\log^* n}$ (Theorem~\ref{thm:alpha2}).

Our first goal is to derive upper and lower bounds on $\alpha(n,d)$ for fixed $d\geq 3$. We prove that the multi-dimensional Hales-Jewett theorem~\cite{FK91} yields $\alpha(n,3)\in o(n)$ (Theorem~\ref{thm:alpha3}). But for $d\geq 4$, only the trivial upper bound $\alpha(n,d)\in O(n)$ is known. We establish lower bounds $\alpha(n,d)\gtrsim (n\log n)^{1/d}$ in a dual setting of hyperplane arrangements in $\mathbb{R}^d$ as described below.

\medskip

\paragraph{\bf Independent sets of hyperplanes.}
For a finite set $\bigh$ of hyperplanes in $\mathbb{R}^d$, Bose et~al.~\cite{BCCHKLT13} defined
a hypergraph $G(\bigh)$ with vertex set $\bigh$ such that the hyperplanes containing the facets of each cell of the arrangement of $\bigh$ form a hyperedge in $G(\bigh)$. A subset $S\subseteq \bigh$ of hyperplanes is called \emph{independent} if it is an independent set of $G(\bigh)$; that is, if no cell of the arrangement of $\bigh$ is bounded by hyperplanes in $S$ only. Denote by $\beta(\bigh)$ the maximum size of an independent set of $\bigh$, and let $\beta (n, d) := \min \{ \beta (\bigh) : |\bigh| = n \}$.

The following relation between $\alpha (n, d)$ and $\beta (n, d)$ was observed by Ackerman et~al.~\cite{AP12} in the case $d=2$.

\begin{lemma}[Ackerman et~al.~\cite{AP12}]\label{lem:AP12}
For $d\geq 2$ and $n\in \mathbb{N}$, we have $\beta (n, d) \leq \alpha(n, d)$.
\end{lemma}

\begin{proof}
For every set $P$ of $n$ points in $\mathbb{R}^d$ in which at most $d+1$ points lie on a
hyperplane, we construct a set $\bigh$ of $n$ hyperplanes in $\mathbb{R}^d$ such that
$\beta (\bigh )\leq \alpha (P)$. Consider the set $\bigh_0$ of hyperplanes obtained from
$P$ by duality. Since at most $d+1$ points of $P$ lie on a hyperplane, at most $d+1$ hyperplanes
in $\bigh_0$ have a common intersection point. Perturb the hyperplanes in $\bigh_0$ so that every
$d+1$ hyperplanes that intersect forms a simplicial cell, and denote by $\bigh$ the resulting set of hyperplanes. An independent subset of hyperplanes corresponds to a subset in general position in $P$. Thus $\alpha (P)\geq \beta (\bigh)$.
\end{proof}

Ackerman et~al.~\cite{AP12}  proved that $\beta (n, 2) \gtrsim \sqrt{n \log n}$, using a result by Kostochka et~al.~\cite{KMV} on independent sets in bounded-degree hypergraphs. Lemma~\ref{lem:AP12} implies that any improvement on this lower bound would immediately improve F\"uredi's lower bound for $\alpha(n,2)$. We generalize the lower bound to higher dimensions by proving that $\beta(n,d)\gtrsim (n \log n)^{1/d}$ for fixed $d\geq2$ (Theorem~\ref{thm:is}).

\medskip

\paragraph{\bf Subsets either in General Position or in a Hyperplane.}
We also consider a generalization of the first problem, and define $\alpha (n, d, \ell)$, with a slight abuse of notation, to be the largest integer such that every set of $n$ points in $\mathbb{R}^d$ in which at most $\ell$ points lie in a hyperplane contains a subset of $\alpha (n, d, \ell)$ points in general position.
Note that $\alpha (n, d) = \alpha (n, d, d+1)$ with this notation, and every set of $n$ points in $\mathbb{R}^d$ contains $\alpha(n,d,\ell)$ points in general position or $\ell+1$ points in a hyperplane.

Motivated by a question of Gowers~\cite{Gowers-GenPos}, Payne and Wood~\cite{PW13} studied $\alpha(n,2,\ell)$;  that is,
the minimum, taken over all sets of $n$ points in the plane with at most $\ell$ collinear, of the
maximum size a subset in general position. They combine the Szemer\'edi-Trotter Theorem~\cite{ST83} with lower bounds on maximal independent sets in bounded-degree hypergraphs to prove $\alpha (n, 2, \ell)\gtrsim \sqrt{n\log n / \log \ell}$. We generalize their techniques, and show that for fixed $d\geq2$ and all $\ell\lesssim \sqrt{n}$, we have $\alpha (n, d, \ell)\gtrsim (n / \log \ell)^{1/d}$ (Theorem~\ref{thm:main}). It follows that every set of at least $Cq^d\log q$ points in $\mathbb{R}^d$, where $C=C(d)>0$ is a sufficiently large constant, contains $q$ cohyperplanar points or $q$ points in general position (Corollary~\ref{cor:Ramsey3d}).

\section{Subsets in General Position and the Hales-Jewett Theorem}

Let $[k] := \{1, 2, \ldots ,k\}$ for every positive integer $k$.
A subset $S\subseteq [k]^m$ is a \emph{$t$-dimensional combinatorial subspace}
of $[k]^m$ if there exists a partition of $[m]$ into sets $W_1,W_2,\ldots ,W_t, X$ such that $W_1, W_2,\ldots ,W_t$ are nonempty, and $S$ is exactly the set of elements $x\in [k]^m$ for which $x_i=x_j$ whenever $i,j\in W_\ell$ for some $\ell\in [t]$, and $x_i$ is constant if $i\in X$. A one-dimensional combinatorial subspace is called a \emph{combinatorial line}.

To obtain a quantitative upper bound for $\alpha(n,2)$, we combine F\"uredi's argument with the quantitative version of the density Hales-Jewett theorem for $k=3$ obtained in the first polymath project.

\begin{theorem}[Polymath~\cite{polymath}]
\label{thm:dhj3}
The size of the  largest subset of $[3]^m$ without a combinatorial line is $O(3^m/\sqrt{\log^* m})$.
\end{theorem}

\begin{theorem}\label{thm:alpha2}
$\alpha (n, 2) \lesssim n / \sqrt{\log^* n}$.
\end{theorem}
\begin{proof}
Consider the $m$-dimensional grid $[3]^m$ in $\mathbb{R}^m$ and project it onto $\mathbb{R}^2$ using a generic projection; that is, so that three points in the projection are collinear if and only if their preimages in $[3]^m$ are collinear. Denote by $P$ the resulting planar point set and let $n=3^m$. Since the projection is generic, the only collinear subsets of $P$ are projections of collinear points in the original $m$-dimensional grid, and $[3]^m$ contains at most three collinear points. From Theorem~\ref{thm:dhj3}, the largest subset of $P$ with no three collinear points has size at most the indicated upper bound.
\end{proof}

To bound $\alpha (n,3)$, we use the multidimensional version of the density Hales-Jewett Theorem.

\begin{theorem}[see \cite{FK89,polymath}]
\label{thm:dmhj}
For every $\delta>0$ and every pair of positive integers $k$ and $t$, there exists a positive integer
$M := M(k, \delta, t)$ such that for every $m>M$, every subset of $[k]^m$ of density at least $\delta$ contains
a $t$-dimensional subspace.
\end{theorem}

\begin{theorem}\label{thm:alpha3}
$\alpha (n, 3) \in o(n)$.
\end{theorem}
\begin{proof}
Consider the $m$-dimensional hypercube $[2]^m$ in $\mathbb{R}^m$ and project it onto $\mathbb{R}^3$ using a generic
projection. Let $P$ be the resulting point set in $\mathbb{R}^3$ and let $n:=2^m$.
Since the projection is generic, the only coplanar subsets of $P$ are projections of points of the $m$-dimensional grid $[2]^m$ lying in a two-dimensional subspace. Therefore $P$ does not contain more than four coplanar points.
From Theorem~\ref{thm:dmhj} with $k=t=2$, for every $\delta>0$ and sufficiently large $m$,
every subset of $P$ with at least $\delta n$ elements contains $k^t = 4$ coplanar points.
Hence every independent subset of $P$ has $o(n)$ elements.
\end{proof}

We would like to prove $\alpha (n, d) \in o(n)$ for fixed $d$. However, we cannot apply the same
technique, because an $m$-cube has too many co-hyperplanar points, which remain co-hyperpanar in projection. By the multidimensional Hales-Jewett theorem, every constant fraction of vertices of a hi-dimensional hypercube has this property. It is a coincidence that a projection of a hypercube to $\mathbb{R}^d$ works for $d=3$, because $2^{d-1}=d+1$ in that case.

\begin{comment}
However, the same proof yields a bound on the function $\alpha (n, d, \ell, k)$ for specific values
of $\ell$ and $k$.
\begin{theorem}
$\alpha (n, d, 2^{d-1}, 2^{d-1}-1) = o(n)$.
\end{theorem}
\end{comment}

\section{Lower Bounds for Independent Hyperplanes}

We also give a lower bound on $\beta (n, d)$ for $d\geq2$. By a simple charging argument (see Cardinal and Felsner~\cite{CF14}), one can establish that $\beta (n, d)\gtrsim n^{1/d}$. Inspired by the recent result of Ackerman et al.~\cite{AP12}, we improve this bound by a factor of $(\log n)^{1/d}$.

\begin{lemma}
\label{lem:degree}
Let $\bigh$ be a finite set of hyperplanes in $\mathbb{R}^d$.
For every subset of $d$ hyperplanes in $\bigh$, there are are most
$2^d$ simplicial cells in the arrangement of $\bigh$ such that all $d$
hyperplanes contain some facets of the cell.
\end{lemma}
\begin{proof}
A simplicial cell $\sigma$ in the arrangement of $\bigh$ has exactly $d+1$ vertices, and exactly $d+1$ facets.
Any $d$ hyperplanes along the facets of $\sigma$ intersect in a single point, namely at a vertex of $\sigma$.
Every set of $d$ hyperplanes in $\bigh$ that intersect in a single point can contains $d$ facets of at most $2^d$ simplicial cells (since no two such cells can lie on the same side of all $d$ hyperplanes).
\end{proof}

The following is a reformulation of a result of Kostochka et al.~\cite{KMV}, that is similar to the reformulation of Ackerman et al.\cite{AP12} in the case $d=2$.

\begin{theorem}[Kostochka et al.~\cite{KMV}]
\label{thm:rdegree}
Consider an $n$-vertex $(d+1)$-uniform hypergraph $H$ such that every $d$-tuple of vertices is contained in at most $t=O(1)$ edges,
and apply the following procedure:
\begin{enumerate}
\item let $X$ be the subset of vertices obtained by choosing each vertex independently at random with probability $p$, such that
$pn = \left( n/(t \log \log \log n) \right)^{3/(3d-1)}$,
\item remove the minimum number of vertices of $X$ so that the resulting subset $Y$ induces a triangle-free linear\footnote{A hypergraph is \emph{linear} if it has no pair of distinct edges sharing two or more vertices.} hypergraph $H[Y]$.
\end{enumerate}
Then with high probability $H[Y]$ has an independent set of size at least $\left( \frac nt \log \frac nt \right)^{\frac 1d}$.
\end{theorem}

\begin{theorem}\label{thm:is}
For fixed $d\geq2$, we have $\beta (n, d) \gtrsim \left(n \log n\right)^{1/d}$.
\end{theorem}
\begin{proof}
Let $\bigh$ be a set of $n$ hyperplanes in $\mathbb{R}^d$ and consider the $(d+1)$-uniform hypergraph
$H$ having one vertex for each hyperplane in $\bigh$, and a hyperedge of size $d+1$ for each set of $d+1$
hyperplanes forming a simplicial cell in the arrangement of $\bigh$. From Lemma~\ref{lem:degree},
every $d$-tuple of vertices of $H$ is contained in at most $t := 2^d$ edges.
Applying Theorem~\ref{thm:rdegree}, there is a subset $S$ of hyperplanes of size
$\Omega \left( ( (\frac{n}{2^d}) \log (\frac{n}{2^d}))^{1/d}\right)$ such that no simplicial cell
is bounded by hyperplanes of $S$ only.

However, there might be nonsimplicial cells of the arrangement that are bounded by hyperplanes of $S$ only.
Let $p$ be the probability used to define $X$ in Theorem~\ref{thm:rdegree}. It is known~\cite{Hal04}
that the total number of cells in an arrangement of $d$-dimensional hyperplanes is less than $dn^d$.
Hence for an integer $c\geq d+1$, the expected number of cells of size $c$ that are bounded by
hyperplanes of $X$ only is at most
$$
p^c d n^d \leq  \frac{n^{(4 - 3d) c /(3d-1)}} { (2^d \log \log \log n)^{3/(3d-1)}}\cdot d n^d
\lesssim  d n^{(4 - 3d) c  /(3d-1) + d} .
$$
Note that for $c\geq d+2$, the exponent of $n$ satisfies
$$
\frac{(4- 3d) c}{3d-1} + d < 0.
$$
Therefore the expected number of such cells of size at least $d+2$ is vanishing.

On the other hand we can bound the expected number of cells that are of size at most $d$, 
and that are bounded by hyperplanes of $X$ only, where the expectation is again with respect
to the choice of $X$. Note that cells of size $d$ are necessarily 
unbounded, and in a simple arrangement, no cell has size less than $d$. 
The number of unbounded cells in a $d$-dimensional arrangement is $O(dn^{d-1})$~\cite{Hal04}. 
Therefore, the number we need to bound is at most
$$
p^d O(dn^{d-1}) \lesssim n^{(4 - 3d) d  /(3d-1) + d-1}
\lesssim n^{1/(3d - 1)}  = o(n^{1/d}) .
$$

Consider now a maximum independent set $S$ in the hypergraph $H[Y]$, and for each cell that is bounded by hyperplanes of $S$ only, 
remove one of the hyperplane bounding the cell from $S$. Since $S\subseteq X$, the expected number 
of such cells is $o(n^{1/d})$, hence there exists an $X$ for which the number of remaining hyperplanes in $S\subseteq X$ 
is still $\Omega \left( (n\log n)^{1/d}\right)$, and they now form an independent set.
\end{proof}

We have the following coloring variant of Theorem~\ref{thm:is}.

\begin{corollary}
Hyperplanes of a simple arrangement of size $n$ in $\mathbb{R}^d$ for fixed $d\geq2$ can be colored with
$O\big( n^{1-1/d} / (\log n)^{1/d}\big)$ colors so that no cell is bounded by hyperplanes of a single color.
\end{corollary}
\begin{proof}
From Theorem~\ref{thm:is}, there always exists an independent set of hyperplanes of size at 
least $c\big(n\log n \big)^{1/d}$ for some constant $c$, where logarithms are base 2. 
We define a new constant $c'$ such that
$$
c' = \big( \frac 1c + c'\big) 2^{2/d-1} \Leftrightarrow c' = \frac{2^{2/d-1}}{c(1-2^{2/d-1})} .
$$
We now prove that $n$ hyperplanes forming a simple arrangement in $\mathbb{R}^d$ can be colored with
$c' ( n^{1-1/d} / (\log n)^{1/d}\big)$ colors so that no cell is bounded by hyperplanes of a single color. 
We proceed by induction and suppose this holds for $n/2$ hyperplanes. We apply the greedy algorithm and
iteratively pick a maximum independent set until there are at most $n/2$ hyperplanes left. We assign
a new color to each independent set, then use the induction hypothesis for the remaining hyperplanes. 
This clearly yields a proper coloring. 

Since every independent set has size at least $c\big(\frac n2\log\frac n2 \big)^{1/d}$, 
the number of iterations before we are left with at most $n/2$ hyperplanes is at most
$$
t\leq \frac{\frac n2}{c\big(\frac n2\log\frac n2 \big)^{1/d}}.
$$  
The number of colors is therefore at most
\begin{eqnarray*}
t + c' \left( \frac{\big(\frac n2\big)^{1-1/d}}{\big(\log \frac n2\big)^{1/d}}  \right)
 & \leq & \frac{\frac n2}{c\big(\frac n2\log\frac n2 \big)^{1/d}} + c' \left( \frac{\big(\frac n2\big)^{1-1/d}}{\big(\log \frac n2\big)^{1/d}}  \right) \\
 & = & \big( \frac 1c + c'\big) \left( \frac{\big(\frac n2\big)^{1-1/d}}{\big(\log \frac n2\big)^{1/d}}  \right) \\
 & \leq & \big( \frac 1c + c'\big) \left( 2^{2/d-1} \frac{ n^{1-1/d} }{ \big(\log n\big)^{1/d} }  \right) \\
 & = & c' \left( \frac{n^{1-1/d}}{\big(\log n\big)^{1/d}}  \right),
\end{eqnarray*}
as claimed. In the penultimate line, we used the fact that $\log \frac n2 > \frac 12 \log n$ for $n>4$.
\end{proof}

\section{Large Subsets in General Position or in a Hyperplane}

We wish to prove the following.

\begin{theorem}
\label{thm:main}
Fix $d\geq 2$. Every set of $n$ points in $\mathbb{R}^d$ with at most $\ell$ cohyperplanar points, where $\ell\lesssim n^{1/2}$, contains a subset of $\Omega\left(\left(n/\log \ell \right)^{1/d}\right)$ points in general position. That is,
$$\alpha (n, d, \ell) \gtrsim \left(n/\log \ell \right)^{1/d}
\mbox{ \rm for }\ell\lesssim\sqrt{n}.$$
\end{theorem}

This is a higher-dimensional version of the result by Payne and Wood~\cite{PW13}.
The following Ramsey-type statement is an immediate corollary.

\begin{corollary}
\label{cor:Ramsey3d}
For fixed $d\geq2$ there is a constant $c$ such that every set of at least $cq^d\log q$ points in $\mathbb{R}^d$ contains $q$ cohyperplanar points or $q$ points in general position.
\end{corollary}

In order to give some intuition about Corollary~\ref{cor:Ramsey3d}, it is worth mentioning an easy
proof when $cq^d\log q$ is replaced by $q\cdot {q\choose d}$. Consider a set of $n = q\cdot {q\choose d}$ points in $\mathbb{R}^d$, and let $S$ be a maximal subset in general position. Either $|S|\geq q$ and we are done, or $S$ spans  ${|S|\choose d}\leq {q\choose d}$ hyperplanes, and, by maximality, every point lies on at least one of these hyperplanes. Hence by the pigeonhole principle, one of the hyperplanes in $S$ must contain at least $n/{q\choose d} =q$ points.

\medskip
We now use known incidence bounds to estimate the maximum number of cohyperplanar $(d+1)$-tuples in a point set. In what follows we consider a finite set $P$ of $n$ points in $\mathbb{R}^d$ such that at most $\ell$ points of $P$ are cohyperplanar, where $\ell := \ell (n) \lesssim n^{1/2}$ is a fixed function of $n$. For $d\geq 3$, a hyperplane $h$ is said to be \emph{$\gamma$-degenerate} if at most $\gamma\cdot |P\cap h|$ points in $P\cap h$ lie on a $(d-2)$-flat. A flat is said to be $k$-rich whenever it contains \emph{at least} $k$ points of $P$. The following is a standard reformulation of the classic Szemer\'edi-Trotter theorem on point-line incidences in the plane~\cite{ST83}.

\begin{theorem}[Szemer\'edi and Trotter~\cite{ST83}]
\label{thm:ST}
For every set of $n$ points in $\mathbb{R}^2$, the number of $k$-rich lines is at most
$$ O \left( \frac{n^2}{k^3} + \frac nk\right).$$
This bound is the best possible apart from constant factors.
\end{theorem}

Elekes and T\'oth proved the following higher-dimensional version, involving an additional non-degeneracy condition.

\begin{theorem}[Elekes and T\'oth~\cite{ET05}]
\label{thm:ET}
For every integer $d\geq 3$, there exist constants $C_d>0$ and $\gamma_d>0$ such that for every set of $n$ points in $\mathbb{R}^d$, the number of $k$-rich $\gamma_d$-degenerate planes is at most
$$ C_d \left( \frac{n^d}{k^{d+1}} + \frac{n^{d-1}}{k^{d-1}}\right) .$$
This bound is the best possible apart from constant factors.
\end{theorem}

We prove the following upper bound on the number of cohyperplanar $(d+1)$-tuples in a point set.

\begin{lemma}
\label{lem:sparse}
Fix $d\geq2$. Let $P$ be a set of $n$ points in $\mathbb{R}^d$ with no more than $\ell$ in a hyperplane, where $\ell\in O(n^{1/2})$. Then the number of cohyperplanar $(d+1)$-tuples in $P$ is $O\left( n^d \log \ell \right)$.
\end{lemma}
\begin{proof}
We proceed by induction on $d\geq2$. The base case $d=2$ was established by Payne and Wood~\cite{PW13}, using the Szemer\'edi-Trotter bound (Theorem~\ref{thm:ST}). We reproduce it here for completeness. We wish to bound the number of collinear triples in a set $P$ of $n$ points in the plane. Let $h_k$ be the number of lines containing exactly $k$ points of $P$. The number of collinear 3-tuples is
\begin{eqnarray*}
\sum_{k=3}^{\ell} h_k {k\choose 3} & \leq & \sum_{k=3}^{\ell} k^2 \sum_{i=k}^{\ell} h_i \\
& \lesssim & \sum_{k=3}^{\ell} k^2 \left( \frac{n^2}{k^{3}} + \frac nk \right) \\
& \lesssim  & n^2\log \ell + \ell^2 n \lesssim  n^2\log \ell.
\end{eqnarray*}

We now consider the general case $d\geq 3$. Let $P$ be a set of $n$ points in $\mathbb{R}^d$, no $\ell$ in a hyperplane, where $n\geq d+2$ and $\ell\lesssim \sqrt{n}$. let $\gamma := \gamma_d>0$
be a constant specified in Theorem~\ref{thm:ET}. We distinguish the following three types of $(d+1)$-tuples:

\medskip\paragraph{\bf\boldmath Type~1: $(d+1)$-tuples contained in some $(d-2)$-flat spanned by $P$.}
Denote by $s_k$ the number of $(d-2)$-flats spanned by $P$ that contain exactly $k$ points of $P$. Project $P$ onto a $(d-1)$-flat in a generic direction to obtain a set of points $P'$ in $\mathbb{R}^{d-1}$. Now $s_k$ is the number of hyperplanes of $P'$ containing exactly $k$
points of $P'$. By applying the induction hypothesis on $P'$, the number of cohyperplanar $d$-tuples is %
$$
\sum_{k=d}^{\ell} s_k {k\choose d} \lesssim n^{d-1/2} \log \ell .
$$
Hence the number of $(d+1)$-tuples of $P$ lying in a $(d-2)$-flat spanned by $P$ satisfies
$$
\sum_{k=d+1}^{\ell} s_k {k\choose d+1} \lesssim \ell n^{d-1} \log \ell \leq n^d\log\ell .
$$

\medskip\paragraph{\bf\boldmath Type~2: $(d+1)$-tuples of $P$ that span a $\gamma$-degenerate hyperplane.}
Let $h_k$ be the number of $\gamma$-degenerate hyperplanes containing exactly $k$ points of $P$. By Theorem~\ref{thm:ET},
\begin{eqnarray*}
\sum_{k=d+1}^{\ell} h_k {k\choose d+1} & \leq & \sum_{k=d+1}^{\ell} k^d \sum_{i=k}^{\ell} h_i \\
& \lesssim & \sum_{k=d+1}^{\ell} k^d \left( \frac{n^d}{k^{d+1}} + \frac{n^{d-1}}{k^{d-1}} \right) \\
& \lesssim & n^d\log \ell + \ell^2 n^{d-1} \lesssim n^d\log \ell .
\end{eqnarray*}

\medskip\paragraph{\bf\boldmath Type~3: $(d+1)$-tuples of $P$ that span a hyperplane that is not $\gamma$-degenerate.}
Recall that if a hyperplane $H$ panned by $P$ is not $\gamma$-degenerate, then more than a $\gamma$ fraction of its points lie in a $(d-2)$-flat $L(H)$. We may assume that $L(H)$ is also spanned by $P$. Consider a $(d-2)$-flat $L$ spanned by $P$ and containing exactly $k$ points of $P$. The hyperplanes spanned by $P$ that contain $L$ partition $P\setminus L$. Let $n_r$ be the number of hyperplanes containing $L$ and exactly $r$ points of $P\setminus L$. We have $\sum_{r=1}^{\ell} n_r r \leq n$.

If a hyperplane $H$ is not $\gamma$-degenerate, contains a $(d-2)$-flat $L=L(H)$ with exactly $k$ points, and $r$ other points of $P$, then $k> \gamma (r+k)$, hence $r< (\frac 1{\gamma}-1)k$. Furthermore, all $(d+1)$-tuples that span $H$ must contain at least one point that is not in $L$. Hence the number of $(d+1)$-tuples that span $H$ is at most $O(rk^d)$. The total number of $(d+1)$-tuples of type~3 that span a hyperplane $H$ with a common $(d-2)$-flat $L=L(H)$ is is therefore at most
$$
\sum_{r=1}^{\ell} n_r r k^d \leq nk^d.
$$
Recall that $s_k$ denotes the number of $(d-2)$-flats containing exactly $k$ points. Summing
over all such $(d-2)$-flats and applying the induction hypothesis yields the following upper bound
on the total number of $(d+1)$-tuples spanning hyperplanes that are not $\gamma$-degenerate:
$$
\sum_{k=d+1}^{\ell} s_k n k^d \lesssim n^d \log \ell .
$$

Summing over all three cases, the total number of cohyperplanar $(d+1)$-tuples is $O(n^d\log \ell)$ as claimed.
\end{proof}

In the plane, Lemma~\ref{lem:sparse} gives an $O(n^2 \log \ell)$ bound for the number of collinear triples in an $n$-element point set with no $\ell$ on a line, where $\ell\in O(\sqrt{n})$. This bound is tight for $\ell=\Theta(\sqrt{n})$ for a $\lfloor \sqrt{n}\rfloor\times \lfloor \sqrt{n}\rfloor$ section of the integer lattice. It is almost tight for $\ell\in \Theta(1)$, Solymosi and Soljakovi\'c~\cite{SS13} recently constructed $n$-element point sets for every constant $\ell$ and $\varepsilon>0$ that contains at most $\ell$ points on a line and $\Omega(n^{1-\varepsilon})$ collinear $\ell$-tuples, hence $\Omega(n^{1-\varepsilon}{\ell\choose 3})\subset \Omega(n^{1-\varepsilon})$ collinear triples.

Armed with Lemma~\ref{lem:sparse}, we now apply the following standard result from hypergraph theory due to Spencer~\cite{S72}.

\begin{theorem}[Spencer~\cite{S72}]
\label{thm:spencer}
Every $r$-uniform hypergraph with $n$ vertices and $m$ edges contains an independent set of size at least
\begin{equation}\label{eq:spencer}
\frac{r-1}{r^{r/(r-1)}} \frac{n}{\left(\frac mn\right)^{1/(r-1)}} .
\end{equation}
\end{theorem}

\begin{proof}[Proof of Theorem~\ref{thm:main}]
We apply Theorem~\ref{thm:spencer} to the hypergraph formed by considering all cohyperplanar $(d+1)$-tuples in a given set of $n$ points in $\mathbb{R}^d$, with no $\ell$ cohyperplanar. Substituting $m\lesssim n^d\log \ell$ and $r=d+1$ in \eqref{eq:spencer}, we get a lower bound
$$
\frac{n} { \left( n^{d-1}\log \ell \right)^{1/d} }
= \left( \frac{n}{\log \ell}\right)^{1/d},
$$
for the maximum size of a subset in general position, as desired.
\end{proof}

\bibliographystyle{plain}
\bibliography{independent_hyperplanes}

\end{document}